\documentclass[oneside,english]{amsart}

\usepackage{hyperref}
 
\usepackage[T1]{fontenc}
\usepackage[latin9]{inputenc}
\usepackage{babel}
\usepackage{refstyle}
\usepackage{amstext}
\usepackage{amsthm}
\usepackage{amssymb}

\makeatletter

\AtBeginDocument{\providecommand\propref[1]{\ref{prop:#1}}}
\AtBeginDocument{\providecommand\thmref[1]{\ref{thm:#1}}}
\RS@ifundefined{subsecref}
  {\newref{subsec}{name = \RSsectxt}}
  {}
\RS@ifundefined{thmref}
  {\def\RSthmtxt{theorem~}\newref{thm}{name = \RSthmtxt}}
  {}
\RS@ifundefined{lemref}
  {\def\RSlemtxt{lemma~}\newref{lem}{name = \RSlemtxt}}
  {}

\numberwithin{equation}{section}
\numberwithin{figure}{section}
\theoremstyle{plain}
\newtheorem{thm}{\protect\theoremname}
  \theoremstyle{remark}
  \newtheorem{rem}{\protect\remarkname}
  \theoremstyle{plain}
  \newtheorem{prop}{\protect\propositionname}
  \theoremstyle{plain}
  \newtheorem*{thm*}{\protect\theoremname}
  \theoremstyle{plain}
  \newtheorem{lem}{\protect\lemmaname}

\makeatother

  \providecommand{\lemmaname}{Lemma}
  \providecommand{\propositionname}{Proposition}
  \providecommand{\remarkname}{Remark}
  \providecommand{\theoremname}{Theorem}
\providecommand{\theoremname}{Theorem}

\begin{document}

\title{Scattering for nls with a potential on the line}

\author{David Lafontaine $^{*}$}

\thanks{$^{*}$ david.lafontaine@unice.fr, Laboratoire de Mathématiques J.A.
Dieudonné, UMR CNRS 7351, Université de Nice - Sophia Antipolis, 06108
Nice Cedex 02 France}

\begin{abstract}
We show the $H^{1}$ scattering for a one dimensional nonlinear Schrödinger
equation with a non-negative, repulsive potential $V$ such that $V,xV\in W^{1,1}$,
and a mass-supercritical non-linearity. We follow the approach of
concentration-compacity/rigidity first introduced by Kenig and Merle.
\end{abstract}

\maketitle
\tableofcontents{}

\section{Introduction}

We consider the following one dimensional defocusing, non linear Schrödinger
equation with a potential

\begin{equation}
i\partial_{t}u+\Delta u-Vu=u|u|^{\alpha},\;u(0)=\varphi\in H^{1}(\mathbb{R}).\label{eq:nlsv}
\end{equation}
If $V\in L^{1}$, $-\Delta+V$ is essentially self-adjoint, so by
Stones theorem the equation is globally well posed in $L^{2}(\mathbb{R})$
and $e^{it(-\Delta+V)}$ is an $L^{2}$-isometry. Goldberg and Schlag
obtained in \cite{MR2096737} the dispersive estimate 
\[
\Vert e^{-it(-\Delta+V)}\psi\Vert_{L^{\infty}}\lesssim\frac{1}{|t|^{\frac{1}{2}}}\Vert\psi\Vert_{L^{1}}
\]
under the assumption that $V$ belongs to $L_{1}^{1}(\mathbb{R})$,
ie $\int_{-\infty}^{\infty}|V(x)|(1+|x|)dx<\infty$, and that $-\Delta+V$
has no resonance at zero energy. In particular, we will consider a
non-negative potential, which always verifies this no-resonance hypothesis
as we will see in Section 2. This estimate gives us usual Strichartz
estimates described below in the paper. Because of the energy conservation
law
\[
E(u(t)):=\frac{1}{2}\int|\nabla u(t)|^{2}+\int V|u(t)|^{2}+\frac{1}{\alpha+2}\int|u(t)|^{\alpha+2}=E(u(0))
\]
the $L^{2}$-well-posedness result extends to the global well-posedness
of the problem (\ref{eq:nlsv}) in $H^{1}(\mathbb{R})$: for every
$\varphi\in H^{1}(\mathbb{R})$, there exists a unique, global solution
$u\in C(\mathbb{R},H^{1}(\mathbb{R}))$ of (\ref{eq:nlsv}). Finally,
let us recall that the mass $M(u(t)):=\int|u(t)|^{2}$ is conserved
too. 

For the mass-supercritical ($\alpha>4$) homogeneous equation
\begin{equation}
i\partial_{t}u+\Delta u=u|u|^{\alpha},\;u(0)=\varphi\in H^{1}(\mathbb{R})\label{eq:nls0}
\end{equation}
it is well known since Nakanishi's paper \cite{MR1726753} that the
solutions \textit{scatter} in $H^{1}(\mathbb{R})$, that is, for every
solution $u\in C(\mbox{\ensuremath{\mathbb{R}},}H^{1}(\mathbb{R}))$
of (\ref{eq:nls0}), there exists a unique couple of data $\psi_{\pm}\in H^{1}(\mathbb{R})$
such that
\[
\Vert u(t)-e^{-it\Delta}\psi_{\pm}\Vert_{H^{1}(\mathbb{R})}\underset{t\rightarrow\pm\infty}{\longrightarrow0}.
\]
Alternative proofs of this result can be found in \cite{MR2518079},
\cite{MR2527809}, \cite{MR2838120} and \cite{MR2576708}.

We prove the scattering of solutions of (\ref{eq:nlsv}) in dimension
one for sufficently regular, non-negative and repulsive potential
$V$.
\begin{thm}
\label{thm:main}Let $\alpha>4$ and $V\in L_{1}^{1}(\mathbb{R})$
be such that $V'\in L_{1}^{1}(\mathbb{R})$. We suppose moreover that
$V$ is non-negative and repulsive: $V\geq0$ and $xV'\leq0$. Then,
every solution $u\in C(\mathbb{R},H^{1}(\mathbb{R}))$ of (\ref{eq:nlsv})
with potential $V$ scatters in $H^{1}(\mathbb{R})$.
\end{thm}
We use the strategy of concentration-compacity/rigidity first introduced
by Kenig and Merle in \cite{MR2257393}, and extented to the intercritical
case by Holmer and Roudenko in \cite{MR2421484}, Duyckaerts, Holmer
and Roudenko in \cite{MR2470397}. In the case of a potential, the
main difficulty is the lack of translation invariance of the equation.
Notice that Hong obtained in \cite{MR2134950} the same result in
the three dimensional case for the focusing equation. However, his
approach cannot be extended to lower dimensions, as it requires endpoint
Strichartz estimates which are not available. Banica and Visciglia
treated in \cite{MR213495011} the case of the non linear Schrödinger
equation with a Dirac potential on the line, and we follow their approach.
The Dirac potential is more singular, but it allows the use of explicit
formulas that are not available in the present more general framework.
\begin{rem}
In dimension one or two, assume that $V$ is smooth and compactly
supported, and such that $\int V<0$. Then the operator $-\Delta+V$
has a negative eigenvalue: as a consequence, the hypothesis of positivity
of $V$ cannot be relaxed as in dimension three, where \cite{MR2134950}
only supposes that the potential has a small negative part, and, in
the same way, the hypothesis of repulsivity, which is needed for the
rigidity, cannot be relaxed to $xV'$ having a small positive part. 
\end{rem}

\begin{rem}
The hypothesis $V,V'\in L^{1}$ are needed to show that the operator
$A=-\Delta+V$ verifies the hypothesis of the abstract profile decomposition
of \cite{MR213495011}, whereas the hypothesis $xV'\in L^{1}$ and
$xV'\leq0$ are needed in the rigidity part.
\end{rem}

\begin{rem}
The same proof holds in dimension two up to the numerology and some
changes in the Hölder inequalities used in Propositions \propref{prof_assumpt},
\propref{decay_lin}, and \propref{decay_duhamel} to deal with the
fact that $H^{1}(\mathbb{R}^{2})$ is not embedded in $L^{\infty}(\mathbb{R}^{2})$.
\end{rem}
~

\begin{rem}
In the focusing, mass-supercritical case
\[
i\partial_{t}u+\Delta u-Vu+u|u|^{\alpha}=0
\]
the same arguments could be used to prove the scattering up to the
natural threshold given by the ground state associated to the equation,
in the spirit of \cite{MR2838120}.
\end{rem}

\subsection*{Notations. }

We will denote by $V$ a potential on the line satisfaying the hypothesis
of \thmref{main}, $\alpha$ will be a real number such that $\alpha>4$.
We set
\[
H^{1}=H^{1}(\mathbb{R}),\ C(H^{1})=C(\mathbb{R},H^{1}(\mathbb{R})),\;L^{p}L^{r}=L^{p}(\mathbb{R},L^{r}(\mathbb{R})),\;L^{p}(I)L^{r}=L^{p}(I,L^{r}(\mathbb{R}))
\]
for any interval $I$ of $\mathbb{R}$. We will denote by $\tau_{y}$
the translation operator defined by $\tau_{y}u=u(\cdot-y)$. Finaly,
we will use $A\lesssim B$ for inequalities of the type $A\leq CB$
where $C$ is a universal constant.

\section{Preliminaries}

From now on, we will fix the four following Strichartz exponents 
\[
r=\alpha+2,\;q=\frac{2\alpha(\alpha+2)}{\alpha^{2}-\alpha-4},\;p=\frac{2\alpha(\alpha+2)}{\alpha+4},\ \gamma=\frac{2\alpha}{\alpha-2}
\]

\subsection{Strichartz estimates}

Recall that we assume all along the paper that $V$ is in $L_{1}^{1}(\mathbb{R})$
and non negative. Goldberg and Schlag obtained in particular in \cite{MR2096737}
the dispersive estimate for the Schrödinger operator $-\Delta+V$
under these assumptions. 

Indeed, they require the hypothesis of absence of resonances at zero
energy. We claim that for $V\geq0$ this hypothesis is satisfied:
by the definition of \cite{MR2096737}, if there is a resonance at
zero, the solutions $u_{\pm}$ of
\begin{equation}
u''=Vu\label{eq:zerener}
\end{equation}
such that $u_{\pm}(x)\rightarrow1$ as $x\rightarrow\pm\infty$ have
a null Wronskian. Therefore $u_{\pm}$ are proportional, so they are
both non trivial bounded solutions of (\ref{eq:zerener}). But such
solutions cannot exist: indeed, if $u$ is such a solution, integrating
(\ref{eq:zerener}) one deduces that $u'$ has limits at $\pm\infty$.
These limits are both zero otherwise $u$ is not bounded. Now, multiplying
(\ref{eq:zerener}) by $u$, integrating it on $[-R,R]$, and letting
$R$ going to infinity, we obtain $\int_{\mathbb{R}}|u'|^{2}+V|u|^{2}=0$.
Therefore $u=0$, a contradiction.
\begin{prop}[Dispersive estimate, \cite{MR2096737}]
Let $V\in L_{1}^{1}(\mathbb{R})$ be such that $V\geq0$. Then, for
all $\psi\in L^{1}(\mathbb{R})$, we have 
\begin{equation}
\Vert e^{-it(-\Delta+V)}\psi\Vert_{L^{\infty}}\lesssim\frac{1}{|t|^{\frac{1}{2}}}\Vert\psi\Vert_{L^{1}}.\label{eq:linflun}
\end{equation}
\end{prop}
Note that, interpolating the previous dispersive estimate (\ref{eq:linflun})
with the mass conservation law, we obtain immediatly for all $a\in[2,\infty]$
\begin{equation}
\Vert e^{it(-\Delta+V)}\psi\Vert_{L^{a}}\lesssim\frac{1}{|t|^{\frac{1}{2}(\frac{1}{a'}-\frac{1}{a})}}\Vert\psi\Vert_{L^{a'}}.\label{eq:dispest}
\end{equation}

Because of (\ref{eq:linflun}), we obtain by the classical $TT^{\star}$
method (see for example \cite{MR1646048}) the Strichartz estimates
\begin{equation}
\Vert e^{-it(-\Delta+V)}\varphi\Vert_{L^{q_{1}}L^{r_{1}}}+\Vert\int_{0}^{t}e^{-i(t-s)(-\Delta+V)}F(s)ds\Vert_{L^{q_{2}}L^{r_{2}}}\lesssim\Vert\varphi\Vert_{L^{2}}+\Vert F\Vert_{L^{q_{3}'}L^{r_{3}'}}\label{eq:adst}
\end{equation}
for all pairs $(q_{i},r_{i})$ satisfying the admissibility condition
in dimension one, that is
\[
\frac{2}{q_{i}}+\frac{1}{r_{i}}=\frac{1}{2}.
\]
We will need moreover the following Strichartz estimates associated
to non admissible pairs:
\begin{prop}[Strichartz estimates]
\label{strichartz}For all $\varphi\in H^{1}$, all $F\in L^{q'}L^{r'}$,
all $G\in L^{q'}L^{r'}$ and all $H\in L^{\gamma'}L^{1}$
\begin{equation}
\Vert e^{-it(-\Delta+V)}\varphi\Vert_{L^{p}L^{r}}\lesssim\Vert\varphi\Vert_{H^{1}}\label{eq:st2}
\end{equation}
\begin{equation}
\Vert e^{-it(-\Delta+V)}\varphi\Vert_{L^{\alpha}L^{\infty}}\lesssim\Vert\varphi\Vert_{H^{1}}\label{eq:st2-1}
\end{equation}
\begin{equation}
\Vert\int_{0}^{t}e^{-i(t-s)(-\Delta+V)}F(s)ds\Vert_{L^{\alpha}L^{\infty}}\lesssim\Vert F\Vert_{L^{q'}L^{r'}}\label{eq:st3}
\end{equation}

\begin{equation}
\Vert\int_{0}^{t}e^{-i(t-s)(-\Delta+V)}G(s)ds\Vert_{L^{p}L^{r}}\lesssim\Vert G\Vert_{L^{q'}L^{r'}}\label{eq:st4}
\end{equation}

\begin{equation}
\Vert\int_{0}^{t}e^{-i(t-s)(-\Delta+V)}H(s)ds\Vert_{L^{p}L^{r}}\lesssim\Vert H\Vert_{L^{\gamma'}L^{1}}.\label{eq:st4-1}
\end{equation}
\end{prop}
\begin{proof}
The estimates (\ref{eq:st2}) $-$ (\ref{eq:st4}) are exactly the
same as $(3.1)-(3.4)$ of \cite{MR2576708}, with the operator $-\Delta+V$
instead of $H_{q}$. As the proof of \cite{MR2576708} relies only
on the admissible Strichartz estimates (\ref{eq:adst}) that are given
by Proposition 1, the same proof holds here. Finally, (\ref{eq:st4-1})
enters on the frame of the non-admissible inhomogeneous Strichartz
estimates of Theorem 1.4 of Foschi's paper \cite{MMR2134950}.
\end{proof}

\subsection{Perturbative results}

We will need the three following classical perturbative results, which
follow immediatly from the previous Strichartz inequalities:
\begin{prop}
\label{pert1}Let $u\in C(H^{1})$ be a solution of (\ref{eq:nlsv}).
If $u\in L^{p}L^{r}$, then $u$ scatters in $H^{1}$. 
\end{prop}

\begin{prop}
\label{pert2}There exists $\epsilon_{0}>0$, such that, for every
data $\varphi\in H^{1}$ such that $\Vert\varphi\Vert_{H^{1}}\leq\epsilon_{0}$,
the corresponding maximal solutions of (\ref{eq:nlsv}) and (\ref{eq:nls0})
both scatter in $H^{1}$.
\end{prop}
\begin{proof}[Proof of Propositions \ref{pert1} and \ref{pert2}]
The proof is the same as for Propositions 3.1 and 3.2 of \cite{MR213495011},
using the Strichartz estimates of our Proposition \ref{strichartz}
instead of their estimates (3.1), (3.2), (3.3), (3.4).
\end{proof}
\begin{prop}
\label{pert3}For every $M>0$ there exists $\epsilon>0$ and $C>0$
such that the following occurs. Let $v\in C(H^{1})\cap L^{p}L^{r}$
be a solution of the following integral equation with source term
$e(t,x)$ 
\[
v(t)=e^{-it(\Delta-V)}\varphi-i\int_{0}^{t}e^{-i(t-s)(\Delta-V)}(v(s)|v(s)|^{\alpha})ds+e(t)
\]
 with $\Vert v\Vert_{L^{p}L^{r}}<M$ and $\Vert e\Vert_{L^{p}L^{r}}<\epsilon$.
Assume moreover that $\varphi_{0}\in H^{1}$ is such that $\Vert e^{-it(\Delta-V)}\varphi_{0}\Vert_{L^{p}L^{r}}<\epsilon$
. Then, the solution $u\in C(H^{1})$ to (\ref{eq:nlsv}) with initial
condition $\varphi+\varphi_{0}$ satisfies
\[
u\in L^{p}L^{r},\quad\Vert u-v\Vert_{L^{p}L^{r}}<C.
\]
 
\end{prop}
\begin{proof}
It is the same as for Proposition 4.7 in \cite{MR2838120}, using
Strichartz estimates (\ref{eq:st4}) instead of Strichartz-type inequality
(4.3) of their paper.
\end{proof}

\section{Profile decomposition}

The aim of this section is to show that we can use the abstract profile
decomposition obtained by \cite{MR213495011}, and inspired by \cite{MR3017992}:
\begin{thm*}[Astract profile decomposition, \cite{MR213495011}]
 Let $A:L^{2}\supset D(A)\rightarrow L^{2}$ be a self adjoint operator
such that:

\begin{itemize}
\item for some positive constants $c,C$ and for all $u\in D(A)$, 
\begin{equation}
c\Vert u\Vert_{H^{1}}^{2}\leq(Au,u)+\Vert u\Vert_{L^{2}}^{2}\leq C\Vert u\Vert_{H^{1}}^{2},\label{eq:ab1}
\end{equation}
\item let $B:D(A)\times D(A)\ni(u,v)\rightarrow(Au,v)+(u,v)_{L^{2}}-(u,v)_{H^{1}}\in\mathbb{C}$.
Then, as $n$ goes to infinity
\begin{equation}
B(\tau_{x_{n}}\psi,\tau_{x_{n}}h_{n})\rightarrow0\quad\forall\psi\in H^{1}\label{eq:ab2}
\end{equation}
as soon as
\[
x_{n}\rightarrow\pm\infty,\quad\sup\Vert h_{n}\Vert_{H^{1}}<\infty
\]
or
\[
x_{n}\rightarrow\bar{x}\in\mathbb{R},\quad h_{n}\underset{H^{1}}{\rightharpoonup}0,
\]
\item let $(t_{n})_{n\geq1}$, $(x_{n})_{n\geq1}$ be sequences of real
numbers, and $\bar{t},\bar{x}\in\mathbb{R}$. Then
\begin{equation}
|t_{n}|\rightarrow\infty\Longrightarrow\Vert e^{it_{n}A}\tau_{x_{n}}\psi\Vert_{L^{p}}\rightarrow0,\quad\forall2<p<\infty,\ \forall\psi\in H^{1}\label{eq:ab3}
\end{equation}
\begin{equation}
t_{n}\rightarrow\bar{t},\ x_{n}\rightarrow\pm\infty\Longrightarrow\forall\psi\in H^{1},\:\exists\varphi\in H^{1},\quad\tau_{-x_{n}}e^{it_{n}A}\tau_{x_{n}}\psi\overset{H^{1}}{\rightarrow}\varphi\label{eq:ab4}
\end{equation}
\begin{equation}
t_{n}\rightarrow\bar{t},\ x_{n}\rightarrow\bar{x}\Longrightarrow\forall\psi\in H^{1},\quad e^{it_{n}A}\tau_{x_{n}}\psi\overset{H^{1}}{\rightarrow}e^{i\bar{t}A}\tau_{\bar{x}}\psi.\label{eq:ab5}
\end{equation}
\end{itemize}
And let $(u_{n})_{n\geq1}$ be a bounded sequence in $H^{1}$. Then,
up to a subsequence, the following decomposition holds
\[
u_{n}=\sum_{j=1}^{J}e^{it_{j}^{n}A}\tau_{x_{n}^{j}}\psi_{j}+R_{n}^{J}\quad\forall J\in\mathbb{N}
\]
where
\[
t_{j}^{n}\in\mathbb{R},\;x_{j}^{n}\in\mathbb{R},\;\psi_{j}\in H^{1}
\]
are such that

\begin{itemize}
\item for any fixed $j$,
\begin{equation}
t_{j}^{n}=0\;\forall n,\quad\text{or}\quad t_{n}^{j}\overset{n\rightarrow\infty}{\rightarrow}\pm\infty\label{eq:dp1}
\end{equation}
\begin{equation}
x_{j}^{n}=0\;\forall n,\quad\text{or\quad}x_{n}^{j}\overset{n\rightarrow\infty}{\rightarrow}\pm\infty,\label{eq:dep2}
\end{equation}
\item orthogonality of the parameters:
\begin{equation}
|t_{j}^{n}-t_{k}^{n}|+|x_{j}^{n}-x_{k}^{n}|\overset{n\rightarrow\infty}{\rightarrow}\infty,\quad\forall j\neq k,\label{eq:dp3}
\end{equation}
\item decay of the reminder:
\begin{equation}
\forall\epsilon>0,\exists J\in\mathbb{N},\quad\limsup_{n\rightarrow\infty}\Vert e^{-itA}R_{n}^{J}\Vert_{L^{\infty}L^{\infty}}\leq\epsilon,\label{eq:dp4}
\end{equation}
 
\item orthogonality of the Hilbert norm:
\begin{equation}
\Vert u_{n}\Vert_{L^{2}}^{2}=\sum_{j=1}^{J}\Vert\psi_{j}\Vert_{L^{2}}^{2}+\Vert R_{n}^{J}\Vert_{L^{2}}^{2}+o_{n}(1),\quad\forall J\in\mathbb{N}\label{eq:dp5}
\end{equation}
\begin{equation}
\Vert u_{n}\Vert_{H}^{2}=\sum_{j=1}^{J}\Vert\tau_{x_{n}^{j}}\psi_{j}\Vert_{H}^{2}+\Vert R_{n}^{J}\Vert_{H}^{2}+o_{n}(1),\quad\forall J\in\mathbb{N}\label{eq:dp6}
\end{equation}
 where $(u,v)_{H}=(Au,v)$, and
\begin{equation}
\Vert u_{n}\Vert_{L^{p}}^{p}=\sum_{j=1}^{J}\Vert e^{it_{j}^{n}A}\tau_{x_{n}^{j}}\psi_{j}\Vert_{L^{p}}^{p}+\Vert R_{n}^{J}\Vert_{L^{p}}^{p}+o_{n}(1),\quad\forall2<p<\infty,\quad\forall J\in\mathbb{N}.\label{eq:dp7}
\end{equation}
\end{itemize}
\end{thm*}
We will see that the self-adjoint operator $A:=-\Delta+V$ verifies
the hypothesis of the previous theorem.
\begin{prop}
\label{prop:prof_assumpt}Let $A:=-\Delta+V$. Then $A$ satisfies
the assumptions (\ref{eq:ab1}), (\ref{eq:ab2}),(\ref{eq:ab3}),(\ref{eq:ab4}),(\ref{eq:ab5}).
\end{prop}
\begin{proof}
\textbf{Assumption (\ref{eq:ab1}).} Because $V$ is positive and
by the Sobolev embedding $H^{1}(\mathbb{R})\hookrightarrow L^{\infty}$,
\[
\Vert u\Vert_{H^{1}}^{2}\leq(Au,u)+\Vert u\Vert_{L^{2}}=\int|\nabla u|^{2}+\int V|u|^{2}+\int|u|^{2}\leq(1+\Vert V\Vert_{L^{1}})\Vert u\Vert_{H^{1}}^{2}
\]
and (\ref{eq:ab1}) holds.

\textbf{Assumption (\ref{eq:ab2}).} We have
\[
B(\tau_{x_{n}}\psi,\tau_{x_{n}}h_{n})=\int V\tau_{x_{n}}\psi\overline{\tau_{x_{n}}h_{n}}.
\]
If $x_{n}\rightarrow\bar{x}\in\mathbb{R},\:h_{n}\underset{H^{1}}{\rightharpoonup}0$,
then $\tau_{x_{n}}\psi\rightarrow\tau_{\bar{x}}\psi$ strongly in
$L^{2}$ and $V\tau_{x_{n}}h_{n}\rightharpoonup0$ weakly in $L^{2}$
(indeed, note that $V\in W^{1,1}(\mathbb{R})\hookrightarrow L^{2}$),
so $B(\tau_{x_{n}}\psi,\tau_{x_{n}}h_{n})\rightarrow0$. Now, let
us assume that $x_{n}\rightarrow\pm\infty$ and $\sup\Vert h_{n}\Vert_{H^{1}}<\infty$.
For example assume that $x_{n}\rightarrow+\infty$. $\psi\in H^{1}(\mathbb{R})$
and therefore decays at infinity: $\epsilon>0$ been fixed, we can
choose $\varLambda>0$ large enough so that 
\[
\sup_{|x|\geq\varLambda}|\psi(x)|\leq\epsilon.
\]
Because $V\in L^{1}$, $\varLambda$ can also be choosen large enough
so that
\[
\int_{|x|\geq\varLambda}|V|\leq\epsilon.
\]
Then, by the Cauchy-Schwarz inequality, and because of the Sobolev
embedding $H^{1}(\mathbb{R})\hookrightarrow L^{\infty}$
\begin{multline*}
|B(\tau_{x_{n}}\psi,\tau_{x_{n}}h_{n})|\leq\Vert h_{n}\Vert_{L^{_{\infty}}}\int|V\tau_{x_{n}}\psi|\\
\leq\sup_{j\geq1}\Vert h_{j}\Vert_{H^{1}}\left(\int_{|x-x_{n}|\geq\varLambda}|V\psi(\cdot-x_{n})|+\int_{|x-x_{n}|\leq\varLambda}|V\psi(\cdot-x_{n})|\right).
\end{multline*}
Now, let $n_{0}$ be large enough so that for all $n\geq n_{0}$,
$x_{n}\geq2\varLambda$. Then, for all $n\geq n_{0}$ 
\[
|x-x_{n}|\leq\varLambda\Rightarrow|x|\geq\varLambda
\]
and, for all $n\geq n_{0}$
\[
|B(\tau_{x_{n}}\psi,\tau_{x_{n}}h_{n})|\leq M\left(\epsilon\Vert V\Vert_{L^{1}}+\epsilon\Vert\psi\Vert_{L^{\infty}}\right)
\]
so (\ref{eq:ab2}) holds.

\textbf{Assumption (\ref{eq:ab3}).} It is an immediate consequence
of the dispersive estimate and the translation invariance of the $L^{p}$
norms. Indeed, because $H_{0}^{1}(\mathbb{R})=H^{1}(\mathbb{R})$,
if $\epsilon>0$, there exists a $C^{\infty}$, compactly supported
function $\tilde{\psi}$ such that 
\begin{equation}
\Vert\tilde{\psi}-\psi\Vert_{H^{1}}\leq\epsilon.\label{eq:ass31}
\end{equation}
But $\tilde{\psi}\in L^{p'}$, so by the dispersive estimate (\ref{eq:dispest})
\[
\Vert e^{it_{n}A}\tau_{x_{n}}\tilde{\psi}\Vert_{L^{p}}\lesssim\frac{1}{|t_{n}|^{\frac{1}{2}(\frac{1}{p'}-\frac{1}{p})}}\Vert\tau_{x_{n}}\tilde{\psi}\Vert_{L^{p'}}=\frac{1}{|t_{n}|^{\frac{1}{2}(\frac{1}{p'}-\frac{1}{p})}}\Vert\tilde{\psi}\Vert_{L^{p'}}\rightarrow0
\]
as $n\rightarrow\infty$. Therefore, for $n$ big enough
\begin{equation}
\Vert e^{it_{n}A}\tau_{x_{n}}\tilde{\psi}\Vert_{L^{p}}\leq\epsilon.\label{eq:ass32}
\end{equation}
To achieve the proof, note that $e^{itA}f$ verifies
\begin{equation}
\Vert e^{itA}f\Vert_{H^{1}}\lesssim\Vert f\Vert_{H^{1}}.\label{eq:subiso}
\end{equation}
Indeed, as $V$ is positive and in $L^{1}$, by the Sobolev embedding
$H^{1}(\mathbb{R})\hookrightarrow L^{\infty}$ we get
\[
\Vert\nabla f\Vert_{L^{2}}^{2}\leq\Vert(-\Delta+V)^{\frac{1}{2}}f\Vert_{L^{2}}^{2}=\int|\nabla u|^{2}+\int V|u|^{2}\lesssim\Vert f\Vert_{H^{1}}.
\]
So, as $e^{itA}$ commute with $(-\Delta+V)^{\frac{1}{2}}$ and is
an isometry on $L^{2}$, 
\begin{multline*}
\Vert e^{itA}f\Vert_{H^{1}}^{2}\leq\Vert e^{itA}f\Vert_{L^{2}}^{2}+\Vert(-\Delta+V)^{\frac{1}{2}}e^{itA}f\Vert_{L^{2}}^{2}\\
=\Vert e^{itA}f\Vert_{L^{2}}^{2}+\Vert e^{itA}(-\Delta+V)^{\frac{1}{2}}f\Vert_{L^{2}}^{2}\\
=\Vert f\Vert_{L^{2}}^{2}+\Vert(-\Delta+V)^{\frac{1}{2}}f\Vert_{L^{2}}^{2}\lesssim\Vert f\Vert_{H^{1}}^{2}.
\end{multline*}
Now, because of the Sobolev embedding $H^{1}\hookrightarrow L^{p}$
we obtain using (\ref{eq:ass31}), (\ref{eq:ass32}) and (\ref{eq:subiso}),
for $n$ big enough 
\begin{eqnarray*}
\Vert e^{it_{n}A}\tau_{x_{n}}\psi\Vert_{L^{p}} & \leq & \Vert e^{it_{n}A}\tau_{x_{n}}(\psi-\tilde{\psi})\Vert_{L^{p}}+\Vert e^{it_{n}A}\tau_{x_{n}}\tilde{\psi}\Vert_{L^{p}}\\
 & \lesssim & \Vert e^{it_{n}A}\tau_{x_{n}}(\psi-\tilde{\psi})\Vert_{H^{1}}+\Vert e^{it_{n}A}\tau_{x_{n}}\tilde{\psi}\Vert_{L^{p}}\\
 & \lesssim & \Vert\psi-\tilde{\psi}\Vert_{H^{1}}+\Vert e^{it_{n}A}\tau_{x_{n}}\tilde{\psi}\Vert_{L^{p}}\leq2\epsilon
\end{eqnarray*}
which achieves the proof of (\ref{eq:ab3}).

\textbf{Assumption (\ref{eq:ab4}).} We will show that
\[
t_{n}\rightarrow\bar{t},\:x_{n}\rightarrow\pm\infty\;\Rightarrow\:\Vert\tau_{-x_{n}}e^{it_{n}(-\Delta+V)}\tau_{x_{n}}\psi-e^{-i\bar{t}\Delta}\psi\Vert_{H^{1}}\rightarrow0
\]
and hence (\ref{eq:ab4}) will hold with $\varphi=e^{-i\bar{t}\Delta}\psi$.
As $\tau_{x_{n}}$ is an $H^{1}$ isometry and commute with $e^{-i\bar{t}\Delta}$,
it is sufficient to show that, if $t_{n}\rightarrow\bar{t}$ and $x_{n}\rightarrow\pm\infty$,
we have
\[
\Vert e^{it_{n}(-\Delta+V)}\tau_{x_{n}}\psi-e^{-i\bar{t}\Delta}\tau_{x_{n}}\psi\Vert_{H^{1}}\rightarrow0.
\]
For example, if $x_{n}\rightarrow+\infty$. Let us first remark that,
as $\tau_{x_{n}}$ commutes with $e^{-i\bar{t}\Delta}$ and $e^{-it_{n}\Delta}$
, is an $H^{1}$ isometry, and because $e^{-it\Delta}\psi\in C(H^{1})$
\[
\Vert e^{-i\bar{t}\Delta}\tau_{x_{n}}\psi-e^{-it_{n}\Delta}\tau_{x_{n}}\psi\Vert_{H^{1}}=\Vert e^{-i\bar{t}\Delta}\psi-e^{-it_{n}\Delta}\psi\Vert_{H^{1}}\rightarrow0.
\]
Hence, decomposing
\begin{multline*}
e^{it_{n}(-\Delta+V)}\tau_{x_{n}}\psi-e^{-i\bar{t}\Delta}\tau_{x_{n}}\psi=\left(e^{it_{n}(-\Delta+V)}\tau_{x_{n}}\psi-e^{-it_{n}\Delta}\tau_{x_{n}}\psi\right)\\
+\left(e^{-it_{n}\Delta}\tau_{x_{n}}\psi-e^{-i\bar{t}\Delta}\tau_{x_{n}}\psi\right)
\end{multline*}
 we see that it is sufficent to show that 
\begin{equation}
\Vert e^{it_{n}(-\Delta+V)}\tau_{x_{n}}\psi-e^{-it_{n}\Delta}\tau_{x_{n}}\psi\Vert_{H^{1}}\rightarrow0.\label{eq:l1}
\end{equation}

Note that $e^{-it\Delta}\tau_{x_{n}}\psi-e^{it(-\Delta+V)}\tau_{x_{n}}\psi$
is a solution of the following linear Schrödinger equation with zero
initial data
\[
i\partial_{t}u-\Delta u+Vu=Ve^{-it\Delta}\tau_{x_{n}}\psi.
\]
Therefore, by the inhomogenous Strichartz estimates, as $(4,\infty)$
is admissible in dimension one, and because the translation operator
commutes with $e^{-it\Delta}$, we have for $n$ large enough so that
$t_{n}\in(0,\bar{t}+1)$ 
\begin{multline*}
\Vert e^{it_{n}(-\Delta+V)}\tau_{x_{n}}\psi-e^{-it_{n}\Delta}\tau_{x_{n}}\psi\Vert_{L^{2}}\\
\leq\Vert e^{it(-\Delta+V)}\tau_{x_{n}}\psi-e^{-it\Delta}\tau_{x_{n}}\psi\Vert_{L^{\infty}(0,\bar{t}+1)L^{2}}\leq\Vert Ve^{-it\Delta}\tau_{x_{n}}\psi\Vert_{L^{\frac{4}{3}}(0,\bar{t}+1)L^{1}}\\
=\Vert(\tau_{-x_{n}}V)e^{-it\Delta}\psi\Vert_{L^{\frac{4}{3}}(0,\bar{t}+1)L^{1}}\leq(\bar{t}+1)^{\frac{3}{4}}\Vert(\tau_{-x_{n}}V)e^{-it\Delta}\psi\Vert_{L^{\infty}(0,\bar{t}+1)L^{1}}.
\end{multline*}
Hence, estimating in the same manner the gradient of these quantities,
it is sufficient to obtain (\ref{eq:l1}) to show that, as $n$ goes
to infinity
\begin{equation}
\Vert(\tau_{-x_{n}}V)e^{-it\Delta}\psi\Vert_{L^{\infty}(0,\bar{t}+1)W^{1,1}}\rightarrow0.\label{eq:vw11}
\end{equation}

Let us fix $\epsilon>0$. $e^{-it\Delta}\psi\in C([0,\bar{t}+1],H^{1})$
and the functions of $H^{1}(\mathbb{R})$ vanish at infinity, so,
using the compacity in time, there exists $\varLambda>0$ such that
\[
\Vert e^{-it\Delta}\psi\Vert_{L^{\infty}(0,\bar{t}+1)L^{\infty}(|x|\geq\varLambda)}\leq\epsilon.
\]
On the other hand, as $V\in L^{1}$, $\varLambda$ can also be taken
large enough so that
\[
\int_{|x|\ge\varLambda}|V(x)|dx\leq\epsilon.
\]
Let $n_{0}$ be large enough so that for all $n\geq n_{0}$, $x_{n}\geq2\varLambda$.
Then, for $n\geq n_{0}$ 
\[
|x+x_{n}|\leq\varLambda\Rightarrow|x|\geq\varLambda
\]
and for all $t\in(0,\bar{t}+1)$ and all $n\geq n_{0}$ we obtain
\begin{eqnarray*}
\Vert(\tau_{-x_{n}}V)e^{-it\Delta}\psi\Vert_{L^{1}} & = & \int_{|x+x_{n}|\geq\varLambda}|V(\cdot+x_{n})e^{-it\Delta}\psi|+\int_{|x+x_{n}|\leq\varLambda}|V(\cdot+x_{n})e^{-it\Delta}\psi|\\
 & \leq & \epsilon\Vert e^{-it\Delta}\psi\Vert_{L^{\infty}(0,\bar{t}+1)L^{\infty}}+\epsilon\Vert V\Vert_{L^{1}}\\
 & \leq & C(\bar{t},\psi,V)\epsilon
\end{eqnarray*}
thus $\Vert(\tau_{-x_{n}}V)e^{-it\Delta}\psi\Vert_{L^{\infty}(0,\bar{t}+1)L^{1}}\rightarrow0.$
With the same argument, because $V'\in L^{1}$, we can show that $\Vert(\tau_{-x_{n}}V)'e^{-it\Delta}\psi\Vert_{L^{\infty}(0,\bar{t}+1)L^{1}}\rightarrow0$.
To obtain (\ref{eq:vw11}), it only remain to show that
\[
\Vert\tau_{-x_{n}}V(e^{-it\Delta}\psi)'\Vert_{L^{\infty}(0,\bar{t}+1)L^{1}}\rightarrow0.
\]
To this purpose, let $\tilde{\psi}$ be a $C^{\infty}$, compactly
supported function such that (recall that we are in dimension one)
\[
\Vert\psi-\tilde{\psi}\Vert_{H^{1}}\leq\epsilon.
\]
We have, by the Cauchy-Schwarz inequality
\begin{eqnarray*}
\Vert\tau_{-x_{n}}V(e^{-it\Delta}\psi)'\Vert_{L^{1}} & \leq & \Vert\tau_{-x_{n}}V(e^{-it\Delta}\tilde{\psi})'\Vert_{L^{1}}+\Vert\tau_{-x_{n}}V(e^{-it\Delta}(\psi-\tilde{\psi}))'\Vert_{L^{1}}\\
 & \leq & \Vert\tau_{-x_{n}}V(e^{-it\Delta}\tilde{\psi})'\Vert_{L^{1}}+\Vert V\Vert_{L^{2}}\Vert(e^{-it\Delta}(\psi-\tilde{\psi}))'\Vert_{L^{2}}\\
 & \leq & \Vert\tau_{-x_{n}}V(e^{-it\Delta}\tilde{\psi})'\Vert_{L^{1}}+\epsilon\Vert V\Vert_{L^{2}}
\end{eqnarray*}
where $V\in L^{2}$ because of the Sobolev embedding $W^{1,1}(\mathbb{R})\hookrightarrow L^{2}(\mathbb{R})$.
Then, as $(e^{-it\Delta}\tilde{\psi})'\in H^{1}$, $\Vert\tau_{-x_{n}}V(e^{-it\Delta}\tilde{\psi})'\Vert_{L^{\infty}(0,\bar{t}+1)L^{1}}$
can be estimated as $\Vert(\tau_{-x_{n}}V)e^{-it\Delta}\psi\Vert_{L^{\infty}(0,\bar{t}+1)L^{1}}$,
so (\ref{eq:vw11}) holds and the proof of (\ref{eq:ab4}) is completed.

\textbf{Assumption (\ref{eq:ab5}).} We decompose
\[
e^{it_{n}A}\tau_{x_{n}}\psi-e^{i\bar{t}A}\tau_{\bar{x}}\psi=(e^{it_{n}A}\tau_{x_{n}}\psi-e^{it_{n}A}\tau_{\bar{x}}\psi)+(e^{it_{n}A}\tau_{\bar{x}}\psi-e^{i\bar{t}A}\tau_{\bar{x}}\psi).
\]
On the one hand, using the estimate (\ref{eq:subiso}) 
\[
\Vert e^{it_{n}A}\tau_{x_{n}}\psi-e^{it_{n}A}\tau_{\bar{x}}\psi\Vert_{H^{1}}\lesssim\Vert\tau_{x_{n}}\psi-\tau_{\bar{x}}\psi\Vert_{H^{1}}\underset{n\rightarrow\infty}{\longrightarrow0}
\]
by the Lebesgue's dominated convergence theorem. On the other hand,
\[
\Vert e^{it_{n}A}\tau_{\bar{x}}\psi-e^{i\bar{t}A}\tau_{\bar{x}}\psi\Vert_{H^{1}}\underset{n\rightarrow\infty}{\longrightarrow0}
\]
because $e^{i\cdot A}\tau_{\bar{x}}\psi\in C(H^{1})$, and the last
assumption is verified.
\end{proof}

\section{Non linear profiles }

In this section, we will see that for a data which escapes to infinity,
the solutions of (\ref{eq:nlsv}) and (\ref{eq:nls0}) are the same,
in the sense given by the three following Propositions.

Propositions \propref{decay_lin}, \propref{decay_duhamel} and \propref{tx_inf}
are the analogous of Propositions 3.4 and 3.6 of \cite{MR213495011}.
The non linear Schrödinger equation with a dirac potential is more
singular, but it allows the use of explicit formulas that are not
available in the present more general framework.
\begin{prop}
\label{prop:decay_lin}Let $\psi\in H^{1}$, $(x_{n})_{n\geq1}\in\mathbb{R}^{\mathbb{N}}$
be such that $|x_{n}|\rightarrow\infty$. Then, up to a subsequence
\begin{equation}
\Vert e^{-it\Delta}\tau_{x_{n}}\psi-e^{-it(\Delta-V)}\tau_{x_{n}}\psi\Vert_{L^{p}L^{r}}\rightarrow0\label{eq:dec1}
\end{equation}
as $n\rightarrow\infty$.
\end{prop}
\begin{proof}
Up to a subsequence, we can assume that $x_{n}\rightarrow+\infty$
or $x_{n}\rightarrow-\infty$. Let us assume for example $x_{n}\rightarrow+\infty$.

As a first step, we will show that

\begin{equation}
\sup_{n\in\mathbb{N}}\Vert e^{it(-\Delta+V)}\tau_{x_{n}}\psi\Vert_{L^{p}(T,\infty)L^{r}}\rightarrow0\label{eq:loctemp}
\end{equation}
as $T\rightarrow\infty$. Pick $\epsilon>0$. There exists a $C^{\infty}$,
compactly supported function $\tilde{\psi}$ such that 
\[
\Vert\tilde{\psi}-\psi\Vert_{H^{1}}\leq\epsilon.
\]
By Strichartz estimates
\[
\Vert e^{it(-\Delta+V)}(\tau_{x_{n}}\tilde{\text{\ensuremath{\psi}}}-\tau_{x_{n}}\psi)\Vert_{L^{p}L^{r}}\lesssim\Vert\tau_{x_{n}}\tilde{\text{\ensuremath{\psi}}}-\tau_{x_{n}}\psi\Vert_{H^{1}}=\Vert\tilde{\psi}-\psi\Vert_{H^{1}}\leq\epsilon.
\]
On the other hand, as $\tau_{x_{n}}\tilde{\text{\ensuremath{\psi}}}\in L^{r'}$
the dispersive estimate (\ref{eq:dispest}) gives us
\[
\Vert e^{it(-\Delta+V)}\tau_{x_{n}}\tilde{\text{\ensuremath{\psi}}}\Vert_{L^{r}}\lesssim\frac{1}{|t|^{\frac{1}{2}(\frac{1}{r'}-\frac{1}{r})}}\Vert\tau_{x_{n}}\tilde{\text{\ensuremath{\psi}}}\Vert_{L^{r'}}=\frac{1}{|t|^{\frac{1}{2}(1-\frac{2}{r})}}\Vert\tilde{\psi}\Vert_{L^{r'}}
\]
but $\frac{p}{2}(1-\frac{2}{r})=\frac{\alpha^{2}}{\alpha+4}>1$ and
$t\rightarrow\frac{1}{|t|^{\frac{1}{2}(1-\frac{2}{r})}}\in L^{p}(1,\infty)$.
So, there exists $T>0$ such that 
\[
\sup_{n\in\mathbb{N}}\Vert e^{it(-\Delta+V)}\tau_{x_{n}}\tilde{\text{\ensuremath{\psi}}}\Vert_{L^{p}(|t|\geq T)L^{r}}\leq\epsilon.
\]
Taking $\tau_{x_{n}}\psi=\tau_{x_{n}}\tilde{\text{\ensuremath{\psi}}}+(\tau_{x_{n}}\psi-\tau_{x_{n}}\tilde{\text{\ensuremath{\psi}}})$,
we then obtain for $T>0$ large enough
\[
\sup_{n\in\mathbb{N}}\Vert e^{it(-\Delta+V)}\tau_{x_{n}}\psi\Vert_{L^{p}(|t|\geq T)L^{r}}\lesssim\epsilon
\]
and (\ref{eq:loctemp}) holds.

To obtain (\ref{eq:dec1}), we are now reduced to show that for $T>0$
fixed

\[
\Vert e^{-it\Delta}\tau_{x_{n}}\psi-e^{it(-\Delta+V)}\tau_{x_{n}}\psi\Vert_{L^{p}(0,T)L^{r}}\rightarrow0
\]
as $n\rightarrow\infty$. Let $\epsilon>0$. $e^{-it\Delta}\tau_{x_{n}}\psi-e^{it(-\Delta+V)}\tau_{x_{n}}\psi$
is a solution of the following linear Schrödinger equation with zero
initial data
\[
i\partial_{t}u-\Delta u+Vu=Ve^{-it\Delta}\tau_{x_{n}}\psi.
\]
So, by the inhomogenous Strichartz estimate (\ref{eq:st4-1}) 
\begin{eqnarray*}
\Vert e^{-it\Delta}\tau_{x_{n}}\psi-e^{it(-\Delta+V)}\tau_{x_{n}}\psi\Vert_{L_{t}^{p}(0,T)L^{r}} & \lesssim & \Vert Ve^{-it\Delta}\tau_{x_{n}}\psi\Vert_{L_{t}^{\gamma'}(0,T)L^{1}}\\
 & \lesssim & T^{\frac{1}{\gamma'}}\Vert Ve^{-it\Delta}\tau_{x_{n}}\psi\Vert_{L^{\infty}(0,T)L^{1}}\\
 &  & =T^{\frac{1}{\gamma'}}\Vert(\tau_{-x_{n}}V)e^{-it\Delta}\psi\Vert_{L^{\infty}(0,T)L^{1}}
\end{eqnarray*}
because the translation operator $\tau_{x_{n}}$ commutes with the
propagator $e^{-it\Delta}$. But 
\[
\Vert(\tau_{-x_{n}}V)e^{-it\Delta}\psi\Vert_{L^{\infty}(0,T)L^{1}}\underset{n\rightarrow\infty}{\longrightarrow}0
\]
as seen in the proof of Proposition \ref{prop:prof_assumpt}, point
(\ref{eq:ab4}).
\end{proof}
\begin{prop}
\label{prop:decay_duhamel}Let $\psi\in H^{1}$, $(x_{n})_{n\geq1}\in\mathbb{R}^{\mathbb{N}}$
be such that $|x_{n}|\rightarrow\infty$, $U\in C(H^{1})\cap L^{p}L^{r}$
be the unique solution to (\ref{eq:nls0}) with initial data $\psi$,
and $U_{n}(t,x):=U(t,x-x_{n})$. Then, up to a subsequence

\begin{equation}
\Vert\int_{0}^{t}e^{-i(t-s)\Delta}\left(U_{n}|U_{n}|^{\alpha}\right)(s)ds-\int_{0}^{t}e^{-i(t-s)(\Delta-V)}\left(U_{n}|U_{n}|^{\alpha}\right)(s)ds\Vert_{L^{p}L^{r}}\rightarrow0\label{eq:dec2}
\end{equation}
as $n\rightarrow\infty$. 
\end{prop}
\begin{proof}
We follow the same spirit of proof as for Proposition \ref{prop:decay_lin}.
We begin to show that
\begin{equation}
\sup_{n\in\mathbb{N}}\Vert\int_{0}^{t}e^{-i(t-s)(\Delta-V)}\left(U_{n}|U_{n}|^{\alpha}\right)(s)ds\Vert_{L^{p}([T,\infty))L^{r}}\rightarrow0\label{eq:dec2-2}
\end{equation}
as $T$ goes to infinity.

We decompose
\begin{multline*}
\Vert\int_{0}^{t}e^{-i(t-s)(\Delta-V)}\left(U_{n}|U_{n}|^{\alpha}\right)(s)ds\Vert_{L^{p}([T,\infty))L^{r}}\leq\Vert\int_{0}^{T}e^{-i(t-s)(\Delta-V)}\left(U_{n}|U_{n}|^{\alpha}\right)(s)ds\Vert_{L^{p}([T,\infty))L^{r}}\\
+\Vert\int_{T}^{t}e^{-i(t-s)(\Delta-V)}\left(U_{n}|U_{n}|^{\alpha}\right)(s)ds\Vert_{L^{p}([T,\infty))L^{r}}
\end{multline*}
where, by the inhomogenous Strichartz estimates
\[
\Vert\int_{T}^{t}e^{-i(t-s)(\Delta-V)}\left(U_{n}|U_{n}|^{\alpha}\right)(s)ds\Vert_{L^{p}([T,\infty)L^{r})}\leq\Vert U_{n}|U_{n}|^{\alpha}\Vert_{L^{q'}([T,\infty)L^{r'})}=\Vert U|U|^{\alpha}\Vert_{L^{q'}([T,\infty)L^{r'})}
\]
and, by the Hölder inequality 
\[
\Vert U|U|^{\alpha}\Vert_{L^{q'}([T,\infty)L^{r'})}\leq\Vert U\Vert_{L^{p}([T,\infty)L^{r})}^{\alpha+1}\underset{T\rightarrow\infty}{\longrightarrow}0
\]
independently of $n$. On the other hand, by the dispersive estimate
(\ref{eq:dispest})
\begin{multline*}
\Vert\int_{0}^{T}e^{-i(t-s)(\Delta-V)}\left(U_{n}|U_{n}|^{\alpha}\right)(s)ds\Vert_{L^{p}[T,\infty)L^{r}}\\
\leq\Vert\int_{0}^{T}\Vert e^{-i(t-s)(\Delta-V)}\left(U_{n}|U_{n}|^{\alpha}\right)(s)\Vert_{L^{r}}ds\Vert_{L^{p}([T,\infty))}\\
\lesssim\Vert\int_{0}^{T}(t-s)^{-\frac{1}{2}(1-\frac{2}{r})}\Vert\left(U_{n}|U_{n}|^{\alpha}\right)(s)\Vert_{L^{r'}}ds\Vert_{L^{p}([T,\infty))}\\
=\Vert\int_{0}^{T}(t-s)^{-\frac{1}{2}(1-\frac{2}{r})}\Vert\left(U|U|^{\alpha}\right)(s)\Vert_{L^{r'}}ds\Vert_{L^{p}([T,\infty))}\\
\leq\Vert\int_{\mathbb{R}}|t-s|^{-\frac{1}{2}(1-\frac{2}{r})}\Vert\left(U|U|^{\alpha}\right)(s)\Vert_{L^{r'}}ds\Vert_{L^{p}([T,\infty))}\longrightarrow0
\end{multline*}
as $T$ goes to infinity. Indeed, note that by the Hardy-Littlewood-Sobolev
inequality
\[
\Vert\int_{\mathbb{R}}|t-s|^{-\frac{1}{2}(1-\frac{2}{r})}\Vert\left(U|U|^{\alpha}\right)(s)\Vert_{L^{r'}}ds\Vert_{L^{p}}\lesssim\Vert U|U|^{\alpha}\Vert_{L^{q'}L^{r'}}\leq\Vert U\Vert_{L^{p}L^{r}}^{\alpha+1}<\infty
\]
so (\ref{eq:dec2-2}) holds. The same estimate is obviously valid
for the propagator $e^{-it\Delta}$.

It remains to show that for $T>0$ fixed, 
\[
\Vert\int_{0}^{t}e^{-i(t-s)\Delta}\left(U_{n}|U_{n}|^{\alpha}\right)ds-\int_{0}^{t}e^{-i(t-s)(\Delta-V)}\left(U_{n}|U_{n}|^{\alpha}\right)ds\Vert_{L^{p}(0,T)L^{r}}\rightarrow0
\]
as $n\rightarrow\infty$. The difference
\[
\int_{0}^{t}e^{-i(t-s)\Delta}\left(U_{n}|U_{n}|^{\alpha}\right)ds-\int_{0}^{t}e^{-i(t-s)(\Delta-V)}\left(U_{n}|U_{n}|^{\alpha}\right)ds
\]
 is the solution of the following linear Schrödinger equation, with
zero initial data
\[
i\partial_{t}u-\Delta u+Vu=V\int_{0}^{t}e^{-i(t-s)\Delta}\left(U_{n}|U_{n}|^{\alpha}\right)ds.
\]
As a consequence, by the Strichartz estimate (\ref{eq:st4-1})
\begin{multline*}
\Vert\int_{0}^{t}e^{-i(t-s)\Delta}\left(U_{n}|U_{n}|^{\alpha}\right)ds-\int_{0}^{t}e^{-i(t-s)(\Delta-V)}\left(U_{n}|U_{n}|^{\alpha}\right)ds\Vert_{L^{p}(0,T)L^{r}}\\
\lesssim\Vert V\int_{0}^{t}e^{-i(t-s)\Delta}\left(U_{n}|U_{n}|^{\alpha}\right)ds\Vert_{L^{\gamma'}(0,T)L^{1}}\\
\lesssim T^{\frac{1}{\gamma'}}\Vert(\tau_{-x_{n}}V)\int_{0}^{t}e^{-i(t-s)\Delta}\left(U|U|^{\alpha}\right)ds\Vert_{L^{\infty}(0,T)L^{1}}.
\end{multline*}
But $\int_{0}^{t}e^{-i(t-s)\Delta}(U|U|^{\alpha})ds\in C([0,T],H^{1})$
and the functions of $H^{1}(\mathbb{R})$ vanish at infinity, so there
exists $\varLambda>0$ such that 
\[
\Vert\int_{0}^{t}e^{-i(t-s)\Delta}\left(U|U|^{\alpha}\right)ds\Vert_{L^{\infty}(0,T)L^{\infty}(|x|\geq\varLambda)}\leq\epsilon
\]
so
\[
\Vert(\tau_{-x_{n}}V)\int_{0}^{t}e^{-i(t-s)\Delta}\left(U|U|^{\alpha}\right)ds\Vert_{L^{\infty}(0,T)L^{1}}\underset{n\rightarrow\infty}{\longrightarrow}0
\]
in the same way as in the proof of Proposition \ref{prop:prof_assumpt},
point (\ref{eq:ab4}) .
\end{proof}
\begin{prop}
\label{prop:tx_inf}Let $\psi\in H^{1}$, $(x_{n})_{n\geq1},\,(t_{n})_{n\geq1}\in\mathbb{R}^{\mathbb{N}}$
be such that $|x_{n}|\rightarrow\infty$ and $t_{n}\rightarrow\pm\infty$,
$U$ be a solution to (\ref{eq:nls0}) such that
\[
\Vert U(t)-e^{-it\Delta}\psi\Vert_{H^{1}}\underset{t\rightarrow\pm\infty}{\longrightarrow}0
\]
and $U_{n}(t,x):=U(t-t_{n},x-x_{n})$. Then, up to a subsequence
\begin{equation}
\Vert e^{-i(t-t_{n})\Delta}\tau_{x_{n}}\psi-e^{-i(t-t_{n})(\Delta-V)}\tau_{x_{n}}\psi\Vert_{L^{p}L^{r}}\rightarrow0\label{eq:dec1-1}
\end{equation}
and

\begin{equation}
\Vert\int_{0}^{t}e^{-i(t-s)\Delta}\left(U_{n}|U_{n}|^{\alpha}\right)ds-\int_{0}^{t}e^{-i(t-s)(\Delta-V)}\left(U_{n}|U_{n}|^{\alpha}\right)ds\Vert_{L^{p}L^{r}}\rightarrow0\label{eq:dec2-1}
\end{equation}
as $n\rightarrow\infty$. 
\end{prop}
\begin{proof}
The proof is the same as for Proposition \ref{prop:decay_lin} and
Proposition \ref{prop:decay_duhamel}, decomposing the time interval
in $\left\{ |t-t_{n}|>T\right\} $ and its complementary.
\end{proof}
Finaly, we will need the following Proposition of non linear scattering:
\begin{prop}
\label{prop:lin_scatt}Let $\varphi\in H^{1}$. Then there exists
$W_{\pm}\in C(H^{1})\cap L_{\mathbb{R}^{\pm}}^{p}L^{r}$, solution
of (\ref{eq:nlsv}) such that
\begin{equation}
\Vert W_{\pm}(t,\cdot)-e^{-it(\Delta-V)}\varphi\Vert_{H^{1}}\underset{t\rightarrow\pm\infty}{\longrightarrow}0\label{eq:nl1}
\end{equation}
moreover, if $t_{n}\rightarrow\mp\infty$ and 
\begin{equation}
\varphi_{n}=e^{-it_{n}(\Delta-V)}\varphi,\quad W_{\pm,n}(t)=W_{\pm}(t-t_{n})\label{eq:nl2}
\end{equation}
then
\begin{equation}
W_{\pm,n}(t)=e^{-it(\Delta-V)}\varphi_{n}+\int_{0}^{t}e^{-i(t-s)(\Delta-V)}(W_{\pm,n}|W_{\pm,n}|^{\alpha})(s)ds+f_{\pm,n}(t)\label{eq:nl3}
\end{equation}
where
\begin{equation}
\Vert f_{\pm,n}\Vert_{L_{\mathbb{R}^{\pm}}^{p}L^{r}}\underset{n\rightarrow\infty}{\longrightarrow0}.\label{eq:nl4}
\end{equation}
\end{prop}
\begin{proof}
The same proof as \cite{MR213495011}, Proposition 3.5, holds, as
it involves only the analogous Strichartz estimates.
\end{proof}

\section{Construction of a critical element}

We have now all the tools to extract a critical element following
the approach of \cite{MR2838120}. Let
\[
E_{c}=\sup\left\{ E>0\ |\ \forall\varphi\in H^{1},\ E(\varphi)<E\Rightarrow\text{the solution of \ensuremath{(1.1)} with data }\varphi\text{ is in }L^{p}L^{r}\right\} .
\]
We will suppose that the critical energy $E_{c}$ is finite, and deduce
the existence of a solution of (\ref{eq:nlsv}) with a relatively
compact flow in $H^{1}$.
\begin{prop}
\label{prop:extr}If $E_{c}<\infty$, then there exists $\varphi_{c}\in H^{1}$,
$\varphi_{c}\neq0$, such that the corresponding solution $u_{c}$
of (\ref{eq:nlsv}) verifies that $\left\{ u_{c}(t),\ t\geq0\right\} $
is relatively compact in $H^{1}$.
\end{prop}
\begin{proof}
Because of Proposition \ref{pert2}, $E_{c}>0$. Therefore, if $E_{c}<\infty$,
there exists a sequence $\varphi_{n}$ of non-zero elements of $H^{1}$,
such that, if we denote by $u_{n}\in C(H^{1})$ the corresponding
solution of (\ref{eq:nlsv}), we have
\[
E(\varphi_{n})\underset{n\rightarrow\infty}{\longrightarrow}E_{c}
\]
and
\[
u_{n}\notin L^{q}L^{r}.
\]
Thanks to the Proposition \propref{prof_assumpt}, we can apply the
abstract profile decomposition of \cite{MR213495011} to the $H^{1}$-bounded
sequence $\varphi_{n}$ and the operator $A=-\Delta+V$. Up to a subsequence,
$\varphi_{n}$ writes, for all $J\in\mathbb{N}$: 
\[
\varphi_{n}=\sum_{j=1}^{J}e^{-it_{j}^{n}(-\Delta+V)}\tau_{x_{j}^{n}}\psi_{j}+R_{n}^{J}.
\]
where $t_{j}^{n},x_{j}^{n},\psi_{j},R_{n}^{J}$ verifies (\ref{eq:dp1})$-$(\ref{eq:dp7}).
From (\ref{eq:dp6}) and (\ref{eq:dp7}), we have
\[
E_{c}\geq\underset{n\rightarrow\infty}{\limsup}\sum_{j=1}^{J}E(e^{-it_{j}^{n}(-\Delta+V)}\tau_{x_{j}^{n}}\psi_{j}).
\]
We show that there is exactly one non trivial profile, that is $J=1$.
By contradiction, assume that $J>1$. To each profile $\psi_{j}$
we associate family of non linear profiles $(U_{j,n})_{n\geq0}$.
Let $j\in\{1\cdots J\}$. We are in exactly one of the following situations:

\begin{enumerate}
\item If $(t_{j}^{n},x_{j}^{n})=(0,0)$. By the orthogonality condition,
notice that this can happen only for one profile. Because $J>1$,
we have $E(\psi_{j})<E_{c}$, so the solution of (\ref{eq:nlsv})
with data $\psi_{j}$ scatters. If this case happens, let $N\in C(H^{1})\cap L^{p}L^{r}$
be this solution, otherwise, we set $N=0$.
\item If $t_{j}^{n}=0$ and $|x_{j}^{n}|\rightarrow\infty$. Let $U_{j}\in C(H^{1})\cap L^{p}L^{r}$
be the unique solution to (\ref{eq:nls0}) with initial data $\psi_{j}$.
We set $U_{n,j}(x,t):=U(x-x_{j}^{n},t)$.
\item If $x_{j}^{n}=0$ and $t_{j}^{n}\rightarrow\pm\infty$. By Proposition
\ref{prop:lin_scatt}, there exists $U_{j}\in C_{\mathbb{R}_{\pm}}(H^{1})\cap L_{\mathbb{R}_{\pm}}^{p}L^{r}$
a solution to (\ref{eq:nlsv}) such that
\[
\Vert U_{j}(t)-e^{-it(\Delta-V)}\psi_{j}\Vert_{H^{1}}\underset{t\rightarrow\pm\infty}{\longrightarrow}0
\]
 and verifying (\ref{eq:nl1}), (\ref{eq:nl2}), (\ref{eq:nl3}),
(\ref{eq:nl4}). We have 
\[
E(U_{j})=\lim_{n\rightarrow\infty}E(e^{-it_{j}^{n}(-\Delta+V)}\tau_{x_{j}^{n}}\psi_{j})<E_{c}
\]
so $U_{j}\in L^{q}L^{r}$. We set $U_{j,n}(t,x):=U_{j}(t-t_{j}^{n},x)$.
\item If $|x_{j}^{n}|\rightarrow\infty$ and $t_{j}^{n}\rightarrow\pm\infty$.
Let $U_{j}\in C(H^{1})\cap L^{p}L^{r}$ be a solution to (\ref{eq:nls0})
such that
\[
\Vert U_{j}(t)-e^{-it\Delta}\psi_{j}\Vert_{H^{1}}\underset{t\rightarrow\pm\infty}{\longrightarrow}0
\]
We set $U_{j,n}(t,x):=U_{j}(t-t_{j}^{n},x-x_{j}^{n})$.
\end{enumerate}
Now, let 
\[
Z_{n,J}:=N+\sum_{j}U_{n,j}.
\]
By the results of the non linear profiles section - Propositions \ref{prop:decay_lin}
and \ref{prop:decay_duhamel} in situation (2), Proposition \ref{prop:tx_inf}
in situation (3) and Proposition \ref{prop:lin_scatt} in situation
(4) -, we have
\begin{multline}
Z_{n,J}=e^{-it(\Delta-V)}(\varphi_{n}-R_{n,J})+\int_{0}^{t}e^{-i(t-s)(\Delta-V)}(N|N|^{\alpha})(s)ds\\
+\sum_{j}\int_{0}^{t}e^{-i(t-s)(\Delta-V)}(U_{j,n}|U_{j,n}|^{\alpha})(s)ds+r_{n,J}\label{eq:extrac_dec}
\end{multline}
with
\[
\Vert r_{n,J}\Vert_{L^{p}L^{r}}\rightarrow0
\]
as $n\rightarrow\infty$. The decomposition (\ref{eq:extrac_dec})
is the same as obtained in the proof of Proposition 4.1 of \cite{MR213495011},
and we therefore obtain the critical element following their proof,
using our perturbative result of Proposition \ref{pert3} instead
of their Proposition 3.3, and the Strichartz inequalities of our Proposition
\ref{strichartz} instead of estimates (3.1), (3.2), (3.3), (3.4)
of their paper.
\end{proof}

\section{Rigidity}

In this section, we will show that the critical solution constructed
in the previous one assuming the fact that $E_{c}<\infty$ cannot
exist.

We will need the following classical result concerning the compact
families of $H^{1}$
\begin{prop}
\label{prop:funccomp}Suppose that $\left\{ u(t),\ t\geq0\right\} $
is relatively compact in $H^{1}$. Then, for any $\epsilon>0$, there
exists $R>0$ such that
\[
\sup_{t\geq0}\int_{|x|\geq R}\left(|\nabla u(t,x)|^{2}+|u(t,x)|^{2}+|u(t,x)|^{\alpha+2}\right)dx\leq\epsilon
\]
\end{prop}
\begin{proof}
Classic, see e.g. \cite{MR2838120}.
\end{proof}
Now, we can show the rigidity Proposition needed to end the proof:
\begin{prop}
\label{prop:ext}Suppose that $u\in C(H^{1})$ is a solution of (\ref{eq:nlsv})
such that $\left\{ u(t),\ t\geq0\right\} $ is relatively compact
in $H^{1}$. Then $u=0$.
\end{prop}
\begin{proof}
By a classical elementary computation, we get the following virial
identities:

\begin{lem}
Let $u\in C(H^{1})$ be a solution to (\ref{eq:nlsv}) and $\chi$
be a compactly supported, regular function. Then

\begin{equation}
\partial_{t}\int\chi|u|^{2}=2\text{Im}\int\chi'u'\bar{u}\label{eq:vir1}
\end{equation}
\begin{equation}
\partial_{t}^{2}\int\chi|u|^{2}=4\int\chi''|u'|^{2}+\frac{2\alpha}{\alpha+2}\int\chi''|u|^{\alpha+2}-2\int\chi'V'|u|^{2}-\int\chi^{(4)}|u|^{2}.\label{eq:vir2}
\end{equation}
\end{lem}
Now, we assume by contradiction that $u\neq0$. Let $\chi\in C_{c}^{\infty}$
be such that $\chi(x)=x^{2}$ for $|x|\leq1$ and $\chi(x)=0$ for
$|x|\geq2$, set $\chi_{R}:=R^{2}\chi(\frac{\cdot}{R})$ and
\[
z_{R}(t)=\int\chi_{R}|u(t)|^{2}
\]
we have, by (\ref{eq:vir1}), the Cauchy-Schwarz inequality and the
conservation of energy 
\begin{equation}
|z'_{R}(t)|\leq2\int|\chi_{R}'||u'||\bar{u}|\leq CE(u)^{\frac{1}{2}}M(u)^{\frac{1}{2}}R.\label{eq:prig1}
\end{equation}
Moreover, by (\ref{eq:vir2}) 
\begin{eqnarray}
z''_{R}(t) & = & 4\int\chi_{R}''|u'|^{2}+\frac{2\alpha}{\alpha+2}\int\chi_{R}''|u|^{\alpha+2}-2\int\chi_{R}'V'|u|^{2}-\int\chi_{R}^{(4)}|u|^{2}\nonumber \\
 & \ge & 8\int_{|x|\leq R}|u'|^{2}+\frac{4\alpha}{\alpha+2}\int_{|x|\leq R}|u|^{\alpha+2}-C\int_{|x|>R}\left(|u|^{2}+|u|^{\alpha+2}+|u'|^{2}\right)\nonumber \\
 &  & -2\int\chi_{R}'V'|u|^{2}-\int\chi^{(4)}|u|^{2}\label{eq:prig2}
\end{eqnarray}
but, because of conservation of the mass 
\begin{equation}
|\int\chi^{(4)}|u|^{2}|\leq\frac{C}{R^{2}}\Vert u(0)\Vert_{L^{2}}\label{eq:prig3}
\end{equation}
and, because $V$ is repulsive (ie $xV'\leq0$), using the Cauchy-Schwarz
inequality, the Sobolev injection $H^{1}\hookrightarrow L^{\infty}$
and the conservation laws 
\begin{eqnarray}
-2\int\chi_{R}'V'|u|^{2} & = & -2\int_{|x|\le R}xV'|u|^{2}+2\int_{|x|>R}\chi_{R}'V'|u|^{2}\nonumber \\
 & \geq & -C\int_{|x|>R}|xV'||u|^{2}\geq-C\Vert xV'\Vert_{L^{1}(|x|>R)}\Vert u\Vert_{L^{\infty}}^{2}\nonumber \\
 & \geq & -C\Vert xV'\Vert_{L^{1}(|x|>R)}\Vert u\Vert_{H^{1}}^{2}\geq-C(u(0))\Vert xV'\Vert_{L^{1}(|x|>R)}.\label{eq:prig4}
\end{eqnarray}
Let $R_{0}$ be large enough so that
\begin{equation}
\int_{|x|\leq R_{0}}|u|^{\alpha+2}\geq\frac{1}{2}\int|u|^{\alpha+2}:=\delta.\label{eq:prig5}
\end{equation}
We have $\delta>0$ because we suppose that $u$ is non zero. For
$R\geq R_{0}$, we obtain combining (\ref{eq:prig2}) with (\ref{eq:prig3}),
(\ref{eq:prig4}), and (\ref{eq:prig5}) 
\begin{equation}
z''_{R}(t)\geq C\left(\delta-\int_{|x|>R}\left(|u|^{2}+|u|^{\alpha+2}+|u'|^{2}\right)-\frac{1}{R^{2}}\Vert u(0)\Vert_{L^{2}}-\Vert xV'\Vert_{L^{1}(|x|>R)}\right).\label{eq:prig6}
\end{equation}
Because $xV'\in L^{1}$ and using the compacity hypothesis combined
with Proposition \ref{prop:funccomp}, there exists $R\geq R_{0}$
large enough so that
\[
\int_{|x|>R}\left(|u|^{2}+|u|^{\alpha+2}+|u'|^{2}\right)+\frac{1}{R^{2}}\Vert u(0)\Vert_{L^{2}}+\Vert xV'\Vert_{L^{1}(|x|>R)}\leq\frac{\delta}{2}
\]
then, (\ref{eq:prig6}) gives
\[
z_{R}''(t)\geq\frac{C\delta}{2}>0.
\]
Integrating this last inequality contradicts (\ref{eq:prig1}) as
$t\rightarrow\infty$.
\end{proof}
We are now in position to end the proof of \thmref{main} :
\begin{proof}[Proof of Theorem \ref{thm:main}]
 If $E_{c}<\infty$, then the Proposition \ref{prop:extr} allows
us to extract a critical element $\varphi_{c}\in H^{1}$, $\varphi_{c}\neq0$,
such that the corresponding solution $u_{c}$ of (\ref{eq:nlsv})
verifies that $\left\{ u_{c}(t),\ t\geq0\right\} $ is relatively
compact in $H^{1}$. By Proposition \ref{prop:ext}, such a critical
solution cannot exist, so $E_{c}=\infty$ and by Proposition \ref{pert1},
all the solutions of (\ref{eq:nlsv}) scatter in $H^{1}$.
\end{proof}

\subsection*{Acknowledgments: }

The Author thanks N. Visciglia for having submitted him this problem,
enlighting discussions and his warm welcome in Pisa, J. Zheng for
his helpfull comments, his Ph.D advisor F. Planchon for his disponibility
and advices, and the referee for his carefull reading and his constructive
criticism.

\bibliographystyle{IEEEtranS}
\bibliography{scat_pot_lafontaine_3}

\end{document}